\documentclass{amsart}

\usepackage{color}
\usepackage{amsmath}
\usepackage{amsfonts}
\usepackage{amssymb,enumerate}
\usepackage{enumitem}
\usepackage{amsthm}
\usepackage{comment}
\usepackage[all]{xy}
\usepackage{tikz}
\usetikzlibrary{arrows}
\usetikzlibrary{positioning}
\usepackage{hyperref}
\usepackage[nospace,noadjust]{cite}

\theoremstyle{definition}
\newtheorem{lem}{Lemma}[section]
\newtheorem{cor}[lem]{Corollary}
\newtheorem{prop}[lem]{Proposition}
\newtheorem{thm}[lem]{Theorem}

\newtheorem{defn}[lem]{Definition}

\renewcommand{\geq}{\geqslant}
\renewcommand{\leq}{\leqslant}

\numberwithin{equation}{section}

\begin{document}
\author{Jim Coykendall}
\address{Department of Mathematical Sciences\\
	Clemson University\\
	Clemson, SC 29634}
\email[J.~Coykendall]{jcoyken@clemson.edu}
\author{Stacy Trentham}
\address{Department of Science and Mathematics\\
	Northern State University\\
	Aberdeen, SD 57401}
\email[S.~Trentham]{stacy.trentham@northern.edu}

\title{Spontaneous Atomicity for Polynomial Rings with Zero-Divisors}
\begin{abstract}
In this paper, we show that it is possible for a commutative ring with identity to be non-atomic (that is, there exist non-zero nonunits that cannot be factored into irreducibles) and yet have a strongly atomic polynomial extension. In particular, we produce a commutative ring with identity, $R$, that is antimatter (that is, $R$ has no irreducibles whatsoever) such that $R[t]$ is strongly atomic. What is more, given any nonzero nonunit $f(t)\in R[t]$ then there is a factorization of $f(t)$ into irreducibles of length no more than $\text{deg}(f(t))+2$.
\end{abstract}

\maketitle

\section{Introduction and Background}

The last two and a half decades have seen a renaissance in the study of the theory of factorization. The main focus of this research has been in the theater of integral domains, but much work has also been done in the more general setting of commutative rings with identity; for example, the interested reader should consult the papers \cite{DDAZ},\cite{DDAZ2}, and \cite{DDAZ3}.

Even for factorization in integral domains, rather surprising effects can occur. For example, Roitman has produced an example of an atomic domain, $R$, whose polynomial extension $R[t]$ is not atomic (\cite{R1}). Of course, for domains it is the case that if $R[t]$ is atomic, then $R$ must be atomic, but even now, the subtle interplay of atomicity between a domain and its polynomial extension is not completely understood. Perhaps at least as suprising is Roitman's result showing that for the conditions ``$R$ is atomic" and ``$R[[x]]$ is atomic" neither one implies the other (\cite{R2}).

The intent of this note is to provide a companion to the Roitman papers \cite{R1} and \cite{R2}, and to provide a cautionary tale of the subtleties of factorization without the assumption of ``integral domain". We will provide an example of a non-atomic (in fact, with no irreducibles whatsoever) commutative ring with identity, $R$, whose polynomial extension $R[t]$ {\it is} atomic (and is, in fact, strongly atomic).

 We first recall the distinction between ``atom'' and ``strong atom'' in a ring with zero divisors (we note that we will be using the terminology ``(strong) irreducible'' and ``(strong) atom'' interchangably).

\begin{defn}
Let $R$ be a commutative ring with identity. We say that $a\in R$ is an atom if $a=bc$ implies that $a$ is associated to either $b$ or $c$ (in the sense that $(a)=(b)$ or $(a)=(c)$). We say that $a\in R$ is a strong atom if $a=bc$ implies that $a$ is strongly associated to either $b$ or $c$ (in the sense that either $b$ or $c$ is a unit in $R$).
\end{defn}

We say that a ring is (strongly) atomic if every nonzero nonunit is a product of (strong) atoms.

These two notions are distinct (see \cite{DDAZ} for example). But a simple example of the distinction occurs in the ring $\mathbb{Z}/6\mathbb{Z}$ where the element $\overline{3}$ is an atom, but not a strong atom (hence $\mathbb{Z}/6\mathbb{Z}$ is atomic, but not strongly atomic).

\section{Preliminaries and the Example}

We first outline the construction of the ring that we will consider throughout this paper.

Let $\mathbb{F}$ be a perfect field of characteristic $p$ and $\{x_1, x_2,\cdots, x_n,\cdots\}$ be a countable collection of indeterminates. We first define the domain $T$ as follows:

\[
T:=\mathbb{F}[x_1^{\alpha_1}, x_2^{\alpha_2},\cdots, x_n^{\alpha_n},\cdots]
\]

\noindent where the exponents $\alpha_i\in\mathbb{Q}^+\bigcup\{0\}$ range over the non-negative rationals for all $i\geq 1$.
 
We now define the ideal 

\[
I:=\langle\{\prod_{i=1}^\infty x_i^{\beta_i}\}\rangle
\]

\noindent where $\beta_i=0$ for all but finitely many $i$, and $\sum_{i=1}^\infty\beta_i>1$ (essentially $I$ is the ideal generated by monomials of total degree greater than 1).

The ring of our focus will be the ring

\[
R:=T/I.
\]

We record some results concerning the properties of the ring $R$ for later use. We first remark that a typical element (coset) of $T$ can be represented in the form

\[
\epsilon_0+\epsilon_1\overline{X}_1+\cdots +\epsilon_n\overline{X}_n+I
\]

\noindent where each $\epsilon_i\in\mathbb{F}$ and each $\overline{X}_i$ is a monomial from $R$ of the form $\overline{X}_i=x_{i,1}^{a_{i,1}}x_{i,2}^{a_{i,2}}\cdots x_{i,t_i}^{a_{i,t_i}}$. Additionally if $\overline{X}_i=x_{i,1}^{a_{i,1}}x_{i,2}^{a_{i,2}}\cdots x_{i,t_i}^{a_{i,t_i}}$, we say that $\overline{X}_i$ is {\it composed} of the elements $\{x_{i,1}, x_{i,2},\cdots , x_{i, t_i}\}$, and has {\it potential} $\sum_{j=1}^{t_i} a_{i,j}$, and we will write $\text{pot}(\overline{X}_i)=\sum_{j=1}^{t_i} a_{j,t_j}$. If we want to specify a single $x_{i,j}$ we write $\text{pot}_{x_{i,j}}(\overline{X}_i)=a_{i,j}$.

Also in the sequel, we will abuse the notation and represent elements of $R$ as elements of $T$ and suppress the coset notation.

\begin{lem}\label{0dim}
$R$ is $0-$dimensional and quasi-local. In particular, every element of $R$ is either nilpotent or a unit.
\end{lem}

\begin{proof}
Using the notation from above, we let $d_i=\sum_{j=1}^{t_i}a_{i, j}$ be the potential of the monomial $\overline{X}_i$. If $m=\min_{1\leq i\leq n}(d_i)$ then there is an $N\in\mathbb{N}$ such that $p^Nm> 1$ and hence $p^Nd_i> 1$ for all $1\leq i\leq n$.

Note first that if $\epsilon_0=0$ then because the characteristic of $R$ is $p$, we have that

\[
(\epsilon_1\overline{X}_1+\cdots +\epsilon_n\overline{X}_n)^{p^N}=\epsilon_1^{p^N}\overline{X}_1^{p^N}+\cdots +\epsilon_n^{p^N}\overline{X}_n^{p^N}=0.
\]

\noindent Hence $\epsilon_1\overline{X}_1+\cdots +\epsilon_n\overline{X}_n$ is nilpotent. This computation shows that every nonunit is nilpotent and the statements of the lemma follow.
\end{proof}

\begin{prop}\label{nonatomic}
$R$ has no irreducible elements. In particular, $R$ is non-atomic.
\end{prop}

\begin{proof}
If 

\[
\overline{X}=x_1^{a_1}x_2^{a_2}\cdots x_k^{a_k}
\]

\noindent then $\overline{X}^{\frac{1}{p}}=x_1^{\frac{a_1}{p}}x_2^{\frac{a_2}{p}}\cdots x_k^{\frac{a_k}{p}}\in R$.

Since any monomial has a nontrivial (nonassociate) $p^{\text{th}}$ root in $R$, an arbitrary nonzero, nonunit $\epsilon_1\overline{X}_1+\cdots +\epsilon_n\overline{X}_n$ has the nonassociate $p^{\text{th}}$ root $\epsilon_1^{\frac{1}{p}}\overline{X}_1^{\frac{1}{p}}+\cdots +\epsilon_n^{\frac{1}{p}}\overline{X}_n^{\frac{1}{p}}$ (since $\mathbb{F}$ is a perfect field). Hence $R$ contains no irreducibles and is therefore non-atomic.
\end{proof}

The following lemma is straightforward, but will be useful later. The content basically asserts that multiplying the two lowest potential terms (of highest degree) yields a nonzero term in the product of two polynomials (assuming, of course, that the product is not identically $0$).

\begin{lem}\label{survive}
Let $f(t)=\sum_{i=0}^nf_it^i, g(t)=\sum_{i=0}^mg_it^i\in R[t], (f_i,g_i\in R)$ be such that $f(t)g(t)\neq 0$. If $f_j$ contains a monomial that minimizes potential among all monomials in $f(t)$ (and $j$ is maximized in the case that there are multiple monomials of minimal potential) and $g_{j^{\prime}}$ is the analog term for $g(t)$, then the coefficient of $t^{j+j^{\prime}}$ has a (surviving) monomial that is the sum of these minimal potentials.
\end{lem}

\begin{proof}
Suppose that $f_j$ contains a monomial of minimal potential in $f(t)$ (and in the case of multiple minimums, we assume that $j$ is maximal) and $g_{j^{\prime}}$ is the analog for $g(t)$. We will call these monomials (reordering if necessary), $z_1:=x_1^{a_1}x_2^{a_2}\cdots x_k^{a_k}$ and $z_2:=x_1^{b_1}x_2^{b_2}\cdots x_k^{b_k}$ respectively. Here each $a_i, b_i\geq 0$ and $a_i+b_i>0$ for all $1\leq i\leq k$. Since the potential of every term in $f(t)$ (resp. $g(t)$) at degree greater than $j$ (resp. $j^{\prime}$) strictly exceeds $\text{pot}(z_1)$ (resp. $\text{pot}(z_2)$), it remains only to show that there is a a monomial of $\text{pot}(z_1)$ and one of $\text{pot}(z_2)$ whose product cannot be cancelled by the product of two other monomials. 

To, this end we reselect $z_1$ and $z_2$ as follows. Among all monomials of minimal potential, select to maximize $a_1$ (resp. $b_1$). If there are multiple solutions in either case select from among these to maximize $a_2$ (resp. $b_2$). The process terminates for either monomial if a unique maximum is found, and in any case it will terminate for both by the $k^{\text{th}}$ step.

We now observe that if we can find two other monomials of $f_j$ and $g_{j^{\prime}}$ respectively, say $w_1, w_2$ such that $\text{pot}(w_i)=\text{pot}(z_i)$ for $i=1,2$ and $w_1w_2=z_1z_2$ then given our selection of of $z_1$ and $z_2$, we can see that $\text{pot}_{x_i}(w_1)=\text{pot}_{x_i}(z_1), 1\leq i\leq k$. Hence $w_1=z_1$ and $w_2=z_2$, and this establishes the lemma.

\end{proof}

For simplicity, we consider a two-variable analog of the ring we constructed earlier.

\begin{prop}\label{two}
Let $A:=K[x^{\alpha},y^{\beta}]$ where $\alpha, \beta$ range over the non-negative rationals. If $I$ is the ideal generated by all monomials of degree strictly greater than 1, then in the ring $(A/I)[t]$, the polynomial $x+yt$ (abusing the notation) is strongly irreducible.
\end{prop}

\begin{proof}
Let $K$ be a field and let $K[x,y; M]$ be the monoid domain with the indeterminates $x$ and $y$. In $R:=K[x,y;\mathbb{Q}^+]$ we impose the deglex order (see, for example, \cite{adams}) as follows. If $a,b,c,d\in\mathbb{Q}$ are positive we declare that $x^ay^b\prec x^cy^d$ if $a+b<c+d$. If $a+b=c+d$  we again say that $x^ay^b\prec x^cy^d$ if $a<b$ or $c<d$ in the case that $a=b$. So this totally orders the subset of nonzero monomials.

To simplify the argument, we argue from the point of domains as follows. We now suppose that $x+yt=fg +h$ where $f,g\in R[t]$ are two nonunits $\text{mod}(IR[t])$ where is $I$ is the ideal of $R$ generated by all monomials of total degree greater than 1 and $h\in IR[t]$. 

We denote by $\text{min}(f)$ to be the monomial(s) of least degree in $f$. Since we have that $x+yt=fg+h$, it must be the case that $1=\text{deg}(\text{min}(fg))=\text{deg}(\text{min}(f))+\text{deg}(\text{min}(g))$. Letting $a=\text{deg}(\text{min}(f))$ and  $b=\text{deg}(\text{min}(g))$, we take $u,v$ to be two monomials occuring in $f,g$ respectively such that $uv\neq 0\text{mod}( IR[t])$. Note that $1\geq \text{deg}(u)+\text{deg}(v)\geq a+b=1$, forces $\text{deg}(u)=a$ and $\text{deg}(v)=b$.

We now throw out all monomials of $f$ of degree larger than $a$ and all monomials of $g$ with degree larger than $b$ and observe this means that $\text{deg}(u)=a$ for all monomials occuring in $f$ (respectively $\text{deg}(v)=b$ for all monomials occuring in $g$). From this, we conclude that $x+yt=fg$ (hence this factorization is analogous to a factorization in an integral domain). Without loss of generality, we can assume that $g\in R$ and so if we write $f=f_0+f_1t$ then $x=\text{min}(f_0)\text{min}(g)$ and $y=\text{min}(f_1)\text{min}(g)$. Hence $g$ has minimal monomials of the form $x^b$ and $y^b$. We first assume that $0<a,b< 1$. So we consider a (minimal) monomial of $f_0$ (say $z$) that maximizes $\text{deg}_y(z)$ and write $z=x^\alpha y^\beta$ with $\alpha+\beta=a$ and note that the monomial $zy^b=x^\alpha y^{\beta+b}$ must survive in the product and this is our contradiction. Hence either $a=0$ or $b=0$. If $a=0$ then each coefficient of $f$ is a unit in which case the previous argument demonstrates that the degree $0$ term of $fg$ cannot be (just) $x$. Hence we conclude that $b=0$ and hence $g$ is a unit. So we see that $x+yt$ is a strong atom.
\end{proof}

\begin{prop}\label{unit}
Any element of $f(t)\in R[t]$ that has at least one unit coefficient is either a unit or has a factorization into no more than $n$ strong atoms where $n$ is the highest degree term of $f(t)$ that has a unit coefficient.
\end{prop}

\begin{proof}
Let $\mathfrak{M}$ be the maximal ideal of $R$ generated by all the monomials $x_i$ and consider the image of $f(t)$ in the domain $R[t]/\mathfrak{M}[t]\cong\mathbb{F}[t]$, which we will denote by $\overline{f}(t)$. Any factorization of $f(t)\in R[t]$ must have the property that each factor must have at least one unit coefficient. Hence given any decomposition

\[
f(t)=f_1(t)f_2(t)\cdots f_m(t)
\]

\noindent there is a corresponding factorization in $\mathbb{F}[t]$

\[
\overline{f}(t)=\overline{f}_1(t)\overline{f}_2(t)\cdots \overline{f}_m(t).
\]

Note that if $\text{deg}(\overline{f}_i(t))=0$, then $f_i(t)$ is a unit in $R[t]$ and so we will discount this possibility and assume that each $\text{deg}(\overline{f}_i(t))\geq 1$. Since $\mathbb{F}[t]$ is a UFD, this puts an upper bound (namely $\text{deg}(\overline{f}(t))$) on the length of the second decomposition. Since each factor of $f(t)\in R[t]$ must have a (positive degree) unit coefficient, we see that factoring $f(t)$ must terminate after no more than $n$ steps where $n$ is the largest degree term of $f(t)$ that has a unit coefficient. Note this argument also demonstrates that each $f_i(t)$ must be strongly irreducible. Indeed if $f_i(t)=g(t)h(t)$ then $\overline{f}_i(t)=\overline{g}(t)\overline{h}(t)$. Then one of these factors (say $\overline{g}(t)$) is a unit and hence its degree is $0$. So $g(t)$ is a unit (constant term) plus a sum of nilpotent elements (higher degree terms) and hence is a unit in $R[t]$. This establishes the proposition.
\end{proof}

\begin{thm}
The ring $R$ is a non-atomic ring such that $R[t]$ is strongly atomic. What is more, given $f(t)\in R[t]$, a nonzero, nonunit polynomial, one of the following occurs.
\begin{enumerate}
\item If $f(t)$ has a unit coefficient, then $f(t)$ can be written as a product of no more than $n$ strong atoms where $n$ is the highest degree term of $f(t)$ that has a unit coefficient.
\item If $f(t)\in\mathfrak{M}[t]$ has a factorization $f(t)=g(t)h(t)$ with both $g(t), h(t)\in\mathfrak{M}[t]$ then $f(t)$ can be factored into two strong atoms.
\item If  $f(t)\in\mathfrak{M}[t]$ does not have a factorization $f(t)=g(t)h(t)$ with both $g(t), h(t)\in\mathfrak{M}[t]$ then $f(t)$ has a factorization of length no more than $\text{deg}(f(t))+2$ strong atoms.
\end{enumerate}
\end{thm}

\begin{proof}
The fact that $R$ is non-atomic is from Proposition \ref{nonatomic}. To verify that $R[t]$ is strongly atomic, it suffices to show that if $f(t)\in R[t]$ is a nonzero nonunit, then one of the three statements holds. As the first statement is immediate from Proposition \ref{unit}, we focus on the last two.

To verify the last two statements, we build in tandem. First suppose that 

\[
f(t)=g(t)h(t)
\]

\noindent with both $g(t), h(t)\in\mathfrak{M}[t]$. Suppose also that $g(t)$ and $h(t)$ are composed of the elements $x_1,x_2,\cdots, x_m$ and let $y$ and $z$ two other elements (homomorphic images of the original indeterminates $\{x_i\}$) that are distinct from the elements composing $g(t)$ and $h(t)$. Since $y, z$ annihilate all of $\mathfrak{M}$, we have the factorization

\[
f(t)=(g(t)+yt+z)(h(t)+yt+z).
\]

It now suffices to show (without loss of generality) that $g(t)+yt+z$ is strongly irreducible. By way of contradiction assume that $g(t)+yt+z=p(t)q(t)$ and consider the ideal $\mathfrak{N}[t]$ where $\mathfrak{N}$ is the ideal generated by all positive rational powers of the elements $x_i$ with the exception of $y$ and $z$. Passing to the homomorphic image $R[t]/\mathfrak{N}[t]\cong(R/\mathfrak{N})[t]$, we obtain the equation

\[
yt+z=\overline{p}(t)\overline{q}(t).
\]

\noindent But Proposition \ref{two} assures us that $yt+z$ is strongly irreducible. Hence, without loss of generality, $\overline{p}(t)$ is a unit in $(R/\mathfrak{N})[t]$ and since $\mathfrak{N}$ is generated by nilpotents, $p(t)$ is a unit in $R[t]$.

For the final case, we assume that $f(t)$ cannot be factored into a product of two elements from $\frak{M}[t]$. If $f(t)$ is strongly irreducible, then we are done. If not, we can assume $f(t)=g(t)h(t)$ with $g(t)\in(\mathfrak{M}, t)$ and $h(t)\in\mathfrak{M}[t]$. Additionally, it must be the case that $g(t)$ has a term with unit coefficient.

Applying Lemma \ref{survive} to the product $g(t)h(t)$, we see that the highest degree term of $g(t)$ that can possibly be a unit coefficient must occur at degree $m\leq\text{deg}(f(t))$. Hence, Proposition \ref{unit} shows that $g(t)$ can be factored into no more than $m\leq\text{deg}(f(t))$ strongly irreducible factors. 

Once we have produced $g(t)$ that has the largest possible maximal degree term with a unit coefficient, we see that $h(t)$ cannot be decomposed with a factor that is not in $\mathfrak{M}[t]$ (lest we would have the ability produce a factor of $f(t)$ with unit term at a higher degree level than the one in $g(t)$) . So if $h(t)$ is not irreducible, we apply the previously proved statement of this theorem to see that $h(t)$ can be decomposed into two strong irreducibles. Putting it all together, $f(t)$ has a factorization into no more than $\text{deg}(f(t))+2$ strong irreducibles.

\end{proof}

It is interesting to point out that the fact that our example is ``full of'' nilpotents should not be too surprising given the following observation (made possible by an interesting observation by the referee).
\begin{lem}
Let $R$ be reduced and $a\in R$ be a nonzero nonunit. If $a=fg$ with $f,g\in R[t]$, and $f$ is a strong atom in $R[t]$, then $g\in R$
\end{lem}

\begin{proof}
We assume that $g\notin R$ and let $c$ be the leading coefficient of $g$. If $P$ is any prime ideal of $R$, we consider the reduction to $(R/P)[t]\cong R[t]/PR[t]$. Since $(R/P)[t]$ is a domain, we must conclude that $cf\in PR[t]$. Hence $cf$ is in every prime ideal of $R[t]$ and so must be nilpotent. As $R$ (and hence $R[t]$) is reduced, it must be the case that $cf=0$. So $f=f(1+ct)$. Since $f$ is a strong atom, this means that $1+ct$ is a unit in $R[t]$, but $R$ is reduced, so we conclude that $c=0$. Hence $g\in R$.
\end{proof}

\begin{cor}
If $R$ is a reduced ring and $R[t]$ is stongly atomic, then $R$ is (strongly) atomic.
\end{cor}

\begin{proof}
Suppose that $a=f_1f_2\cdots f_n$ with each $f_i\in R[t]$ being strong atoms. By the previous lemma, we can inductively see that each $f_i$ is a (strong) atom of $R$.
\end{proof}

\section*{Acknowledgement}

The authors express gratitude to the referee whose careful reading facilitated an improved version of this paper. Additionally, we thank the referee for the observation that made the last result on strong atomicity possible.

\bibliography{biblio2}{}
\bibliographystyle{plain}

\end{document}